\documentclass[a4paper,intlimits,12pt,reqno]{amsart}
\usepackage[T1]{fontenc}
\usepackage[utf8]{inputenc}

\usepackage{latexsym,amsthm,amsmath,amssymb,mathrsfs,mathtools}
\usepackage{tikz}
\usetikzlibrary{arrows,calc,patterns}
\usepackage{wrapfig}

\usepackage[mathscr]{eucal}

plus 6pt minus 12pt
\newcommand{\A}{\mathcal A}
\newcommand{\B}{\mathcal B}
\newcommand{\C}{\mathcal C}

\newcommand{\bbbr}{\mathbb R}
\newcommand{\bbbb}{\mathbb B}
\newcommand{\bbbs}{\mathbb S}

\newcommand{\N}{\mathbb N}
\newcommand{\R}{\mathbb R}
\newcommand{\overbar}[1]{\mkern 1.7mu\overline{\mkern-1.7mu#1\mkern-1.5mu}\mkern 1.5mu}
\def\Vol{\operatorname{Vol}}
\newcommand{\eps}{\varepsilon}
\newcommand{\tf}{\tilde{f}}

\def\diam{\operatorname{diam}}
\def\dist{\operatorname{dist}}

\def\Log{\operatorname{Log}}
\DeclareMathOperator*{\osc}{osc}
\newtheorem{theorem}{Theorem}
\newtheorem*{theorem*}{Theorem}
\newtheorem{lemma}[theorem]{Lemma}
\newtheorem{corollary}[theorem]{Corollary}
\newtheorem{proposition}[theorem]{Proposition}

\newtheorem{question}{Question}

\usepackage{setspace}

\title[Finite distortion mappings between manifolds]{Finite distortion Sobolev mappings between manifolds are continuous}
\author[Goldstein]{Pawe\l{} Goldstein}
\address{Pawe\l{} Goldstein,\newline \indent Institute of Mathematics,\newline \indent Faculty of Mathematics, Informatics and Mechanics, \newline \indent University of Warsaw \newline \indent Banacha 2, 02-097 Warsaw, Poland}
\email{goldie@mimuw.edu.pl}
\thanks{P.G. was partially supported by NCN grant no. 2012/05/E/ST1/03232 (years 2013-17).}
\author[Haj\l{}asz]{Piotr Haj\l{}asz}
\address{Piotr Haj\l{}asz,\newline \indent Department of Mathematics, University of Pittsburgh, \newline \indent 301 Thackeray Hall, Pittsburgh,
Pennsylvania 15260}
\email{hajlasz@pitt.edu}
\thanks{P.H. was supported by NSF grant DMS-1500647.}

\author[Pakzad]{Mohammad Reza Pakzad}
\address{Mohammad Reza Pakzad, \newline \indent Department of Mathematics, University of Pittsburgh, \newline \indent 301 Thackeray Hall, Pittsburgh,
Pennsylvania 15260}
\email{pakzad@pitt.edu}
\thanks{M.R.P. was partially supported by NSF grant DMS-1210258.}


\begin{document}

\sloppy

\begin{abstract}
We prove that if $M$ and $N$ are Riemannian, oriented $n$-dimensional manifolds without boundary and
additionally $N$ is compact, then Sobolev mappings $W^{1,n}(M,N)$ of finite distortion are continuous.
In particular, $W^{1,n}(M,N)$ mappings with almost everywhere positive Jacobian are continuous.
This result has been known since 1976 in the case of mappings $W^{1,n}(\Omega,\bbbr^n)$, where $\Omega\subset\bbbr^n$
is an open set. The case of mappings between manifolds is much more difficult. 
\end{abstract}

\maketitle

\section{Introduction}

Sobolev functions in $W^{1,p}(\bbbr^n)$ are H\"older continuous when $p>n$. However, if ${1\leq p\leq n}$, Sobolev functions can be discontinuous. For example,
one can easily check that $f(x)=\log|\log|x||\in W^{1,n}$ in a neighborhood of the origin. One can use this example to create a function in $W^{1,n}$
that has $\log|\log|x||$ type singularities located on a dense subset of $\bbbr^n$: we add singularities one by one, with centers located on a countable and dense subset of $\bbbr^n$
and if we do it carefully, we obtain a Cauchy sequence in the $W^{1,n}$ norm, so it converges to a $W^{1,n}$ function. This is to say that, without much effort, one can construct 
a function in $W^{1,n}$ such that on every open subset of $\bbbr^n$ its essential supremum is $+\infty$ and the essential infimum is $-\infty$.  

However,  Vodop'janov and Gol'd{\v{s}}te{\u\i}n \cite{VodoGold} proved that if a mapping $f:\Omega\to\bbbr^n$ of class $W^{1,n}$, defined on a domain $\Omega\subset\bbbr^n$, 
has positive Jacobian, $J_f>0$, almost everywhere, then
$f$ is continuous. Moreover, the mapping is not only continuous, but it also satisfies a certain logarithmic estimate for the modulus of continuity (see Theorem~\ref{thm:osc}).
This estimate (and the idea of the proof) is strictly related to the celebrated Courant-Lebesgue Lemma \cite[Lemma~8.3.5]{Jost}.
In fact, Vodop'janov and Gol'd{\v{s}}te{\u\i}n proved a slightly stronger result that $W^{1,n}$ mappings of finite distortion are continuous. 

We say that a mapping $f\in W^{1,n}(\Omega,\bbbr^n)$ has finite distortion if at almost every point of $\Omega$ either $J_f>0$ or $Df=0$. That is, if the Jacobian $J_f(x)$ is not positive, then necessarily the whole derivative $Df(x)$ must be equal zero.
Thus $W^{1,n}$ mappings with Jacobian positive almost everywhere are of finite distortion.

For applications of the theory of mappings of finite distortion we refer the reader to the monograph \cite{HenclKoskela} and references therein.

A modern version of the result of  Vodop'janov and Gol'd{\v{s}}te{\u\i}n reads as follows:
\begin{theorem}
\label{thm:osc}
If a mapping $f\in W^{1,n}(\Omega,\bbbr^n)$ has finite distortion, then $f$ 
has a continuous representative which
satisfies  the estimate
\begin{equation}
\label{osc1}
\left( \osc_{B(x_{o},r)} f\right)^{n} \leq \frac{C(n)}{\log(R/r)}\int_{B(x_{o},R)} |D f(x)|^{n}\, dx,
\end{equation}
provided $B(x_{o},R)\subset\Omega$ and $0<r<R$.
\end{theorem}
For a proof see e.g. \cite{hajlaszm}. The result in \cite{hajlaszm} is stated for $K$-pseudomonotone mappings. Vodop'janov and Gol'd{\v{s}}te{\u\i}n proved that mappings of finite distortion are weakly monotone and it is easy to verify that weakly monotone mappings are 1-pseudomonotone (see e.g. \cite{hajlaszm}), so the result follows.
Anyway, we will sketch the proof of Theorem~\ref{thm:osc} in Section~\ref{sec:1}.

\begin{corollary} 
If $f\in W^{1,n}(\Omega,\bbbr^{n})$ and $J_{f}>0$ a.e., then $f$ has a continuous representative.
\end{corollary}
The proof uses in an essential way 
the existence of the radial retraction ${P\colon\bbbr^{n}\to \overbar{B(x_{o},r)}}$,
which maps $\bbbr^{n}\setminus B^{n}(x_{o},r)$ onto $S^{n-1}(x_{o},r)=\partial B^{n}(x_{o},r)$.

Using a similar argument one can extend the result to the case of Orlicz-Sobolev mappings:
\begin{theorem}
Assume $f\colon\Omega\to\bbbr^{n}$ has finite distortion and $Df\in L^{n}\Log^{-1}$. Then $f$ has a continuous representative.
\end{theorem}
See \cite[Chapter 2]{HenclKoskela} and \cite[Chapter 7]{IwaniecM} for the case of more general Orlicz-Sobolev spaces.

In this paper we are concerned with Sobolev mappings between manifolds.

Assume $M$ and $N$ are smooth, Riemannian manifolds without boundary, with $N$ being compact.
Assume also that $N$ is 
isometrically embedded in $\bbbr^{k}$ for some $k\in\N$. Then the class of Sobolev mappings $W^{1,p}(M,N)$ 
is defined as
$$
W^{1,p}(M,N)=\left\{u\in W^{1,p}(M,\bbbr^{k})~~\colon~~f(x)\in N \text{ for a.e. }x\in M\right\}.
$$
This is a metric space, with the metric inherited from $W^{1,p}(M,\R^k)$.

If $M$ and $N$ are smooth, oriented, $n$-dimensional Riemannian manifolds without boundary, with $N$ being closed, then we say that
$f\in W^{1,n}(M,N)$ has finite distortion if $J_f(x)>0$ or $Df(x)=0$ at almost every point of $M$.
If in addition, there is a constant $K>0$ such that $|Df|^n\leq K J_f$ almost everywhere, then we say that $f$ is quasiregular.

Note that we need the manifolds to be oriented, as otherwise we would not be able to talk about positive Jacobian.

We want to emphasize that in our definitions of quasiregular mappings and mappings of finite distortion we {\em do not} assume continuity.

As we have seen, Euclidean quasiregular mappings or, more generally, $W^{1,n}$ mappings of finite distortion are continuous.

The class of quasiregular mappings and mappings of finite distortion between manifolds have been studied by many authors, see for 
example \cite{BM,HP,HP2,OP,P} and references therein. In most of the papers the target manifold is compact.
However, the authors always explicitly assume continuity of such mappings. Thus a natural question is whether 
mappings of finite distortion, as defined above, are necessarily continuous. 

\begin{question}
Assume $M$ and $N$ are two smooth, oriented, $n$-dimensional Riemannian manifolds without boundary, with $N$ being compact.
Is it true that a mapping $f\in W^{1,n}(M,N)$ of 
finite distortion is necessarily continuous (i.e. has a continuous representative)?
\end{question}

The Euclidean argument does not easily adapt to this case, because of lack  of a counterpart of the projection $P$. 
This is an essential difference, because, in fact, an analogous question for Orlicz-Sobolev mappings has a negative answer: 
there exists a mapping $f\colon \bbbs^{n}\to \bbbs^{n}$ with $Df\in L^{n}\Log^{-1}$, with positive Jacobian (and thus of finite distortion), 
that is discontinuous, see \cite{HIMO} and Section \ref{sec:3}. The existence of such a counterexample is, in essence, 
caused by the lack of a retraction that could play the role of $P$ in the proof.

However, in the case of $W^{1,n}$ mappings, the answer is in the positive. 
This shows that when the target manifold is compact,
the continuity assumption in papers cited above is redundant.
The main result of the paper reads as follows.
\begin{theorem}
\label{thm:Sob}
Let $M$ and $N$ be smooth,  oriented, $n$-dimensional Riemannian manifolds without boundary and assume additionally that $N$ is compact.
If $f\in W^{1,n}(M,N)$ has finite distortion, then $f$ is continuous. Moreover, for every compact subset $Z\subset M$ there is 
$R_o>0$ depending on $Z$, $N$ and $f$ such that $f$  satisfies the estimate
\begin{equation}
\label{uniform}
\left(\osc_{B(x,r)} f\right)^n \leq \frac{C(n)}{\log (R/r)} \int_{B(x,R)} |D f|^n
\end{equation}
provided $x\in Z$ and $0<r<R<R_o$. 
\end{theorem}
Note that the constant in the inequality depends on $n$ only.
\begin{corollary}
If $M$ and $N$ are as in Theorem~\ref{thm:Sob}, $f\in W^{1,n}(M,N)$ and $J_{f}>0$ almost everywhere in $M$, then $f$ is continuous.
\end{corollary}

Another motivation behind Theorem~\ref{thm:Sob} stems from
a result Hornung and Velcic \cite{HV} who proved continuity of the Gauss map for a $W^{2,2}$-isometric immersion of a $2$-dimensional surface
with smooth positive Gaussian curvature. In a forthcoming paper we will show applications of Theorem~\ref{thm:Sob} to
regularity of higher dimensional Sobolev isometric immersions, see~\cite{GHP}.

One can further ask whether Theorem~\ref{thm:Sob} generalizes to the case of mappings into non-compact manifolds, but, in general, the 
answer is in the negative: the function $f(x)=\log|\log x|$ defines a mapping from a ball in $\bbbr^n$ into the graph of $f$. This mapping
is in $W^{1,n}$, has positive Jacobian, but it is discontinuous at the origin.

The paper is organized in the following way: in Section \ref{sec:1} we present a proof of Theorem~\ref{thm:osc}. 
In Section \ref{sec:2} we prove Theorem \ref{thm:Sob}. Section \ref{sec:3} is devoted to discontinuous Orlicz-Sobolev mappings with 
positive Jacobian. Finally, Section \ref{sec:5} is an appendix with some technical results needed in the paper. 

Notation used in the paper is pretty standard. By $C$ we will denote a generic constant whose value may change within
a single string of estimates. By writing $C(n)$ we indicate the the constant depends on $n$ only.
The Lebesgue measure of a set $A\subset\bbbr^n$ is denoted by $|A|$, while the Lebesgue measure of a subset $A$ of a Riemannian manifold 
will be denoted by $\Vol(A)$. By $\omega_n$ we denote the measure of an $n$-dimensional Euclidean unit ball, therefore the volume of an $(n-1)$-dimensional unit sphere equals $n\omega_n$.
Whenever we consider a restriction of a Sobolev mapping to the boundary of a ball, we mean 
the restriction in the sense of traces. The spaces of continuous and H\"older continuous functions with exponent $\alpha$ will be denoted by $C^0$ and $C^{0,\alpha}$ respectively. The norm of the derivative $|Df|$ is always understood as the operator norm, that is
$|Df(x)|=\sup_{|\xi|=1}|Df(x)\xi|$. The symbol $\lesssim$, used in Section \ref{sec:5}, stands for an inequality up to a constant dependent on the dimension $n$ only.

\section{Euclidean case}
\label{sec:1}

In this section we will sketch the proof of Theorem~\ref{thm:osc}. While the result is known, it is important to see a sketch here, because
it will allow us to understand in what ways the proof of Theorem~\ref{thm:Sob} is different and in what ways it is similar to that of Theorem~\ref{thm:osc}.
Our argument will be somewhat sketched, but the missing details will be explained in the course of the proof of Theorem~\ref{thm:Sob}.

The most important step in the proof is the observation that mappings of finite distortion have the following property:

\begin{equation}
\label{mono}
\osc_{B(x,r)} f\leq 2\osc_{S(x,r)} f
\qquad
\text{for every $x\in\Omega$ and almost all }0<r<\dist (x,\Omega^c).
\end{equation}

Since $W^{1,n}$ functions are H\"older continuous on almost all spheres, the oscillation of $f$ on $S(x,r)$ is understood in the classical
sense: $\osc_{S(x,r)} f=\diam f(S(x,r))$. However, since we do not know yet that $f$ is continuous on $\Omega$, the oscillation of $f$ on the ball is 
understood as the essential oscillation
$$
\osc_{B(x,r)} f = \inf\big\{ \diam f(A):\, \text{$A$ is a full measure subset of $B(x,r)$}\big\}.
$$
In fact, we will prove that if $\overbar{B}$ is a closed ball that contains $f(S(x,r))$ and 
$\diam f(\overbar{B})\leq 2\diam f(S(x,r))$, then
$$
|f^{-1}(\bbbr^n\setminus\overbar{B})\cap B(x,r)|=0,
$$  
that is, almost all points of $B(x,r)$ are mapped into $\overbar{B}$.

Suppose to the contrary that the set
$$
K=f^{-1}(\bbbr^n\setminus\overbar{B})\cap B(x,r)
$$
has positive measure. The mapping $f$ is not constant on $K$, thus the derivative $Df$
cannot be equal zero almost everywhere in $K$ (for a detailed argument see the proof of \eqref{Jf over K}). 
Since $f$ is a mapping of finite distortion, it follows that
$J_f>0$ on a subset of $K$ of positive measure and hence 
$$
\int_K J_f >0.
$$
Let now $P$ be the retraction $P:\bbbr^n\to \overbar{B}$ as described in the Introduction.
Since the Jacobian of the mapping $P\circ f$ equals zero on the set $K$ (because this set is mapped into
an ($n-1$)-dimensional sphere) and $J_{P\circ f}=J_f$ on $B(x,r)\setminus K$, we conclude that
$$
\int_{B(x,r)} J_{P\circ f}< \int_{B(x,r)} J_f.
$$
On the other hand, $f=P\circ f$ on the boundary of the ball $B(x,r)$ (since the boundary is mapped into $\overbar{B}$)
and $W^{1,n}$ mappings that have the same boundary values have equal integrals of the Jacobian \cite[Lemma 4.7.2]{IwaniecM}, so
$$
\int_{B(x,r)} J_{P\circ f}= \int_{B(x,r)} J_f,
$$
which is a contradiction. This proves \eqref{mono}.

By Fubini's theorem, the restriction of $f$ to almost all spheres $S(x_0,r)$  belongs to $W^{1,n}(S(x_0,r),N)$. This and
Morrey's inequality (see Appendix) imply that
$$
\osc_{S(x_0,t)} f\leq C(n)\Bigg( t \int_{S(x_0,t)}|Df|^n\Bigg)^{1/n}
\qquad
\text{for almost all $0<t<R$.}
$$
Hence \eqref{mono} yields
\begin{align*}
& C(n)\int_{B(x_0,R)\setminus B(x_0,r)} |Df|^n  \geq
\int_r^R \left(\osc_{S(x_0,t)} f\right)^n\, \frac{dt}{t} \\ 
& \ \ \geq
2^{-n} \int_r^R \left(\osc_{B(x_0,t)} f\right)^n\frac{dt}{t} \geq
2^{-n}\log(R/r)\left(\osc_{B(x_0,r)} f\right)^n,
\end{align*}
which proves the result.

\section{$W^{1,n}$ mappings between manifolds}
\label{sec:2}

In this section we prove Theorem \ref{thm:Sob}.
Since the result is local in nature, we will state it first under the assumption that $M=\bbbb^n$ is a Euclidean ball (of any radius). Then we will show how the general case of Theorem \ref{thm:Sob} follows from this particular one. 

Recall that the injectivity radius of $N$, denoted by $d_N$, is the supremum of $r>0$ such that
for any $p\in N$, the exponential map $\exp_p:T_p N\to N$ is a diffeomorphism of the ball 
$B(0,r)\subset T_p N$ onto $B(p,r)\subset N$  (for more information about the
exponential map, see \cite{Jost}).

\begin{theorem}
\label{from ball}
Let $N$ be a smooth, compact and oriented, $n$-dimensional Riemannian manifold without boundary. 
Let $d_N$ be the radius of injectivity of $N$ and let $\C_M=(n-1)\pi (n\omega_n)^{-1/n}$
be the constant in Morrey's inequality (Corollary~\ref{cor:morrey}).
Let $\bbbb^n\subset\bbbr^n$ be a Euclidean ball of any radius.
If $f\in W^{1,n}(\bbbb^n,N)$ has finite distortion and satisfies the estimate
$$
\int_{\bbbb^n} |Df|^n<\min\{ \A_N,\B_N\},
$$
where
$$
\A_N =\frac{1}{2}\left(\frac{d_N}{60\C_M}\right)^n\, ,
\quad
\B_N=\inf\{ \Vol(D):\, \text{$D$ is a ball in $N$ of radius $d_N/10$}\},
$$
then $f$ has a continuous representative which satisfies the estimate
\begin{equation}
\label{osc1r}
\left( \osc_{B(x_{o},r)} f\right)^{n} \leq \frac{(6\,\C_M)^n}{\log(R/r)}\int_{B(x_{o},R)} |D f|^{n}\, ,
\end{equation}
provided $B(x_{o},2R)\subset \bbbb^n$ and $0<r<R$. 
\end{theorem}
Note that since the manifold $N$ is compact, the constant $\B_N$ is positive.

Before we prove Theorem~\ref{from ball}, we will show how Theorem~\ref{thm:Sob} follows from it.

Balls in Riemannian manifolds are defined with respect to the Riemannian distance so, by the definition, all balls are geodesic.

For every $x\in M$, the exponential map $\exp_x:T_x M\to M$ is a diffeomorphism of a neighborhood of $0\in T_x M$
onto a neighborhood of $x\in M$, say $\exp_x$ maps a ball $B(0,r_x)$ in $T_x M$ in a diffeomorphic way onto the ball $B(x,r_x)$ in $M$.
By the definition of the exponential map, for every $0<r<r_x$, $\exp_x$ maps $B(0,r)$ onto $B(x,r)$.

Since the derivative $D_x\exp_x:T_x M\to T_x M$, is the identity map, $\exp_x$ restricted to a small neighborhood of $0\in T_x M$ is bi-Lipschitz with the bi-Lipschitz constant $2$. Let $r_x>0$ be the largest radius such that $\exp_x$ restricted to $B(0,r_x)$ is $2$-bi-Lipschitz.

Given a  compact set $Z\subset M$, let
$R_1=\inf\{ r_x: x\in Z\}$. Clearly $R_1>0$ and for any $x\in Z$,
$\exp_x:B(0,R_1)\to B(x,R_1)$ is a $2$-bi-Lipschitz diffeomorphism. Hence
$f\circ\exp_x\in W^{1,n}(B(0,R_1),N)$ and
$$
\int_{B(0,r)} |D (f\circ\exp_x)|^n\leq \tilde{C}(n) \int_{B(x,r)} |Df|^n\quad\text{ for all }0<r<R_1.
$$
Note that since $T_x M$ is equipped with the Euclidean structure, we can identify it with $\bbbr^n$, so $B(0,R_1)$ is a Euclidean ball.

Since $\sup_{x\in Z}\Vol(B(x,r))\to 0$ as $r\to 0$, we have that $R_2>0$, where
$$
R_2=\sup\Big\{ r>0:\, \Big( \tilde{C}(n) \int_{B(x,r)} |Df|^n\Big)\leq\frac{1}{2}\min \{ \A_N,\B_N\}\text{ for all }x\in Z\,\Big\}.
$$
Let $R_3=\min \{ R_1,R_2\}$ and $R_0=R_3/2$. Then for any $x\in Z$, the function
$f\circ\exp_x\in W^{1,n}(B(0,R_3),N)$ satisfies the assumptions of Theorem~\ref{from ball}.
Hence if $x\in Z$ and $0<r<R<R_0$, we have
\begin{align*}
\left(\osc_{B(x,r)} f\right)^n
 =&
\left(\osc_{B(0,r)}(f\circ\exp_x)\right)^n 
\leq
\frac{(6\C_M)^n}{\log(R/r)} \int_{B(0,R)} |D(f\circ\exp_x)|^n \\
& \leq \frac{\tilde{C}(n) (6\C_M)^n}{\log(R/r)}\int_{B(x,R)} |Df|^n\,,
\end{align*}
so Theorem~\ref{thm:Sob} follows. Note that $B(0,2R)\subset B(0,R_3)$ so we could apply Theorem~\ref{from ball}.
Therefore, it remains to prove Theorem~\ref{from ball}.

\begin{proof}[Proof of Theorem~\ref{from ball}]

The exponential map $\exp_p:T_p N\to N$ is a diffeomorphism of the ball 
$B(0,d_N)\subset T_p N$ onto $B(p,d_N)\subset N$. Since $T_pN$ is isometric to $\bbbr^n$, we can assume that the exponential map is defined on $\bbbr^n$.
In particular, if we fix $p\in N$, then for any $0<r<d_N$, the inverse of the exponential map $\exp_p^{-1}$ maps 
$B(p,r)$ onto the Euclidean ball $B(0,r)$ in $\bbbr^n$.

The next result is a crucial step in the proof of Theorem \ref{from ball}. It is a counterpart of property 
\eqref{mono}. Recall that a $W^{1,n}$ function on an ($n-1$)-dimensional sphere has a $C^{0,1/n}$-H\"older continuous representative, see
Section~\ref{sec:mor}.

\begin{proposition}
\label{l11}
Let $B\subset \bbbb^n$ be a ball such that 
\begin{itemize}
\item[(1)] $f|_{\partial B}\in C^{0,1/n}\cap W^{1,n}$,
\item[(2)] $d=\diam f(\partial B)<d_N/10$.
\end{itemize}
Let $D_1$, $D_2$, $D_3$ be concentric balls in $N$ of radii $2d$, $3d$ and $d_N/2$ such that ${f(\partial B)\subset D_1}$.
If 
\begin{equation}
\label{349}
\int_B J_f<\Vol (D_3\setminus D_2),
\end{equation}
then a subset of $B$ of full measure is mapped into $\overbar{D_2}$.That is, the essential oscillation of $f$ on $B$ satisfies
$$
\osc_B f\leq 6\osc_{\partial B} f.
$$
\end{proposition}
\begin{proof}
Suppose to the contrary that the set  $K=B\cap f^{-1}(N\setminus \overbar{D_2})$ has positive measure. We claim that 
\begin{equation}\label{Jf over K}
\int_K J_f>0.
\end{equation}
Let $D^\lambda$, $\lambda>0$, be the ball concentric with $D_2$, of radius $3d+\lambda$. If $\lambda>0$ is 
sufficiently small, then the set $L=B\cap f^{-1}(N\setminus D^\lambda)$ also has positive measure.

Take $\ell$ to be a line in $\R^n$ such that
\begin{itemize}
\item[(1)] $\ell$ is parallel to the $x_1$ axis,
\item[(2)] $\ell\cap L\neq \varnothing$,
\item[(3)] $f$ is absolutely continuous on $\ell\cap \overbar{B}$ (see Corollary \ref{cor:12}).
\end{itemize}
The setting is presented in Figure \ref{fig:1}.

\begin{figure}[ht]
\begin{tikzpicture}[scale=0.57,>=latex']

\filldraw[fill=lightgray!50!white] (4,0) circle (3.7);

\filldraw [draw=darkgray!70!white,fill=white,xshift=-20,xscale=1.2,yscale=1.3] (2.5,0) to [out=90,in=110] (3, 1) to [out=-70,in=-95] (3.9,1.5) to [out=85,in=75] (4.7,0.4) to [out=-115,in=220] (4.9,0.15) to [out=40,in=40] (4.7,-1) -- (4.5,-1.5) -- (4.1,-1) -- (3.8,-1.4) -- (3.7,-1) to [out=100,in=70] (3.3,-1) to [out=250,in=-90] (2.5,0);
\filldraw [draw=darkgray!70!white,pattern=dots,xshift=-20,xscale=1.2,yscale=1.3] (2.5,0) to [out=90,in=110] (3, 1) to [out=-70,in=-95] (3.9,1.5) to [out=85,in=75] (4.7,0.4) to [out=-115,in=220] (4.9,0.15) to [out=40,in=40] (4.7,-1) -- (4.5,-1.5) -- (4.1,-1) -- (3.8,-1.4) -- (3.7,-1) to [out=100,in=70] (3.3,-1) to [out=250,in=-90] (2.5,0);

\filldraw [draw=darkgray, fill=darkgray!60!white] (2.5,0) to [out=90,in=110] (3, 1) to [out=-70,in=-95] (4,1.5) to [out=85,in=75] (4.7,0.4) to [out=-115,in=220] (5,0) to [out=40,in=40] (4.7,-1) -- (4.5,-1.5) -- (4.1,-1) -- (3.9,-1.4) -- (3.8,-1) to [out=100,in=70] (3.3,-1) to [out=250,in=-90] (2.5,0);


\draw[thick,dashed] (0.35,0.5) -- (7.7,0.5);
\node[left] at (0.35,0.5) {$a$};
\node[right] at (7.7,0.5) {$b$};


\begin{scope}[shift={(-2,0)}]
\filldraw [draw=lightgray, fill=lightgray!50!white] (16,1) to [out=45,in=-90] (16,3) to [out=90,in=-135] (18,4) arc [start angle=90, end angle=15,radius=4] to [out=-135,in=90] (21,0) to [out=-90,in=90] (21.5,-1) to [out=-90,in=45] (20,-3) to [out=-135,in=0] (18,-1.5) to [out=180,in=0] (17.5,-1) to [out=180,in=-45] (17.3,-1.5) to [out=135,in=-90] (17,-1) to [out=90,in=-90] (16.5,0) to [out=90,in=-135] (16,1);

\draw (18,0) circle (4.5);
\draw [darkgray,thick] (16.5,0) to [out=90,in=150] (17.5,1.5) to [out=-30,in=160] (18,2) to [out=-20,in=80] (18,0.5) to [out=-100,in=60] (19.5,-0.5) to [out=-120,in=0] (18,-1.5) to [out=180,in=0] (17.5,-1) to [out=180,in=-45] (17.3,-1.5) to [out=135,in=-90] (17,-1) to [out=90,in=-90] (16.5,0);

 \filldraw [draw=darkgray!70!white, pattern=dots] (18,4) arc [start angle=90, end angle=15,radius=4] to [out=45,in=240] (22.04,2) arc [start angle=26.33, end angle=85, radius=4.5] to [out=-120,in=45] (18,4);
 
 \filldraw [draw=darkgray!60!white, fill=darkgray!60!white] (22.04,2) arc [start angle=26.33, end angle=85, radius=4.5] to [out=60,in=-90] (19,7) -- (24,7) -- (24,3);
 \draw[dotted,thick] (19,7) -- (24,7) -- (24,3);
 
\draw[thick] (18,0) circle (4);
\draw[dashed, thick] (17,1.5) to [out=90,in=-160] (18,3) to [out=30,in=-160] (20,5) to [out=20,in=90] (21,4) to [out=-90,in=120] (21,2) to [out=-60,in=60] (21,1.6) to [out=-120,in=120] (21,1.2) to [out=-60,in=60] (21,0.8) to [out=-120,in=120] (20.5,0.4) to [out=-60,in=45] (19.5, -0.45);

\node[above] at (4,-3) {$B$};
\filldraw[fill=white, opacity=0.7,draw=white] (13.7,0.5) rectangle (14.3,-0.5);
\node at (14.1,0) {$D_2$};
\filldraw[fill=white, opacity=0.7,draw=white] (13.2,-0.5) rectangle (13.8,-1.5);
\node at (13.5,-1) {$D^{\lambda}$};
\filldraw[fill=white, opacity=0.7,draw=white] (17.45,-1.7) rectangle (17.62,-2.7);
\node at (17.5,-2.3) {$f(\partial B)$};

\filldraw (17,1.5) circle [radius=0.08];
\node[left,scale=0.9] at (17.1,1.5) {$f(a)$};

\filldraw (19.5,-0.45) circle [radius=0.08];
\node[scale=0.9] at (20,-0.8) {$f(b)$};

\node at (5.6,0) {$L$};
\node at (4.2,0) {$K$};

\node[above, scale=1.1] at (3.3,0.4) {$\ell$};

\node at (20.5,5.5) {$f(\ell\cap\overbar{B})$};
\end{scope}
\end{tikzpicture}\caption{The length of $f(\ell\cap K)$ is at least $2\lambda$, since it travels twice between $D_2$ and $N\setminus D^\lambda$.}\label{fig:1}
\end{figure}
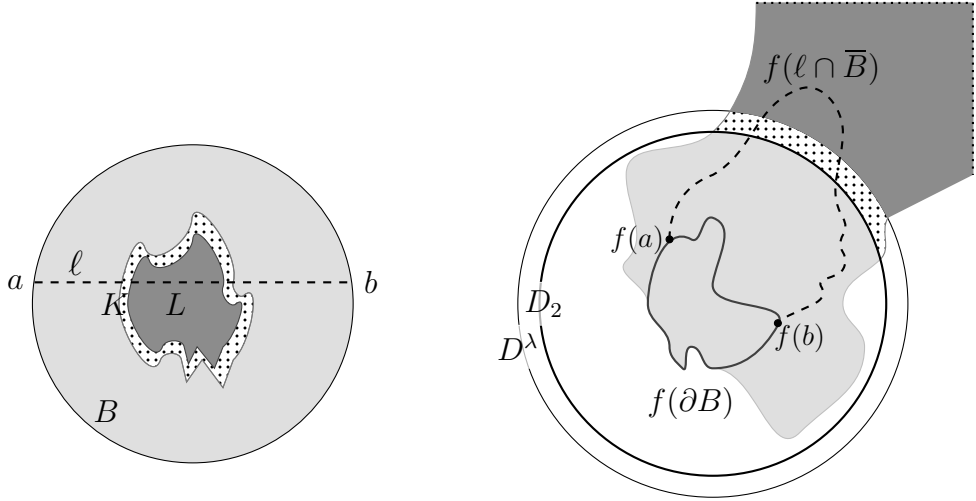
Note that $f(\ell\cap \partial B)\subset D_1$. Let $\ell\cap \partial B=\{a,b\}$. Then $f|_{\ell \cap \overbar{B}}$ is a rectifiable 
curve that starts at $f(a)\in D_1$, travels to $f(\ell\cap L)\subset N\setminus D^\lambda$ and returns to $f(b)\in D_1$. In particular, 
$f(\ell\cap K)$ is a part of the curve $f(\ell\cap \overbar{B})$ that connects the boundary of $D_2$ to points in $N\setminus D^\lambda$
and then returns to the boundary of $D_2$. Thus the length of $f(\ell\cap K)$ equals at least $2\lambda$, because the distance 
between $D_2$ and $N\setminus D^\lambda$ equals $\lambda$. 

Since the (positive) length of $f(\ell\cap K)$ is given by $\int_{\ell\cap K} |\partial f/\partial x_1|$, the set 
$$
{\{x\in \ell\cap K} ~:~ {Df(x)\neq 0\}}
$$ 
has positive linear measure.

The set of lines $\ell$ satisfying the conditions (1), (2) and (3) above has positive $(n-1)$-dimensional measure and by Fubini's theorem 
$$
\left| \{x\in K~~:~~ Df(x)\neq 0\}\right|>0.
$$
The mapping $f$ has finite distortion, thus $J_f(x)>0$ whenever $Df(x)\neq 0$, so \eqref{Jf over K} follows.

Since the behavior of a generic $W^{1,n}$ mapping can be very wild, it will be convenient to approximate $f$ on $B$ by a 
smooth mapping, while keeping the same boundary values.
\begin{lemma}
\label{l2}
Under the assumptions of Proposition~\ref{l11} there is a mapping 
$$
\tf\in C^\infty\cap W^{1,n}(B,N)\cap C^0(\overbar{B},N)
$$ 
such that 
\begin{itemize}
\item[(a)] $f$ and $\tf$  have the same trace on $\partial B$,
\item[(b)] $\int_B J_{\tf} < \Vol (D_3\setminus D_2)$,
\item[(c)] $\int_{B\cap\tf^{-1}(N\setminus D_1)} J_{\tf} >0$.
\end{itemize}
\end{lemma}
\begin{proof}[Proof of Lemma \ref{l2}]
According to \cite[Lemma 6 bis]{Bethuel}, there is a sequence of $C^{\infty}\cap W^{1,n}$ mappings $f_k\colon B\to N$ that 
converge to $f$ in $W^{1,n}(B,N)$ and $f_k |_{\partial B}=f |_{\partial B}$ in the sense of traces.
Since the trace of $f$ on $\partial B$ is continuous, $f_k$ can be chosen to be continuous up to the boundary.

We will show that one can choose $\tf$ to be equal $f_k$ for $k$ sufficiently large.

Note that $J_{f_k}\xrightarrow{L^1} J_f$, which, together with \eqref{349}, implies that
$$
\int_B J_{f_k}<\Vol(D_3\setminus D_2)
$$
for $k$ sufficiently large.

Then,  by \eqref{Jf over K} and convergence of $f_k$,
\begin{equation*}
\int_{K}J_{f_k} \to \int_{K}J_{f}=c>0,
\end{equation*}
thus, for $k$ sufficiently large,
$$
\int_{K}J_{f_k}>\frac{c}{2}>0.
$$
This is already rather close to (c), but the set of integration is wrong, therefore we write
$$
\frac{c}{2}<\int_{K}J_{f_k}
=\int_{K\setminus f_k^{-1}(N\setminus D_1)}J_{f_k}+\int_{K\cap f_k^{-1}(N\setminus D_{1})} J_{f_k}=I_1+I_2.
$$
Consider the set of integration in $I_1$. If 
$$
x\in K\setminus f_k^{-1}(N\setminus D_1)= \left(B\cap f^{-1}(N\setminus \overbar{D_2})\right)\setminus f_k^{-1}(N\setminus D_1),
$$ 
then $f(x)\in N\setminus \overbar{D_2}$, but $f_k(x)\not\in N\setminus D_1$, thus $f_k(x)\in D_1$. 
Therefore $\dist(f(x),f_k(x))>d=\dist(D_1,N\setminus D_2)$, and since $f_k\xrightarrow{L^1}f$, 
the measure of the set
$$
K\setminus f_k^{-1}(N\setminus D_1)\subset\left\{x\in B~~:~~\dist(f_k(x),f(x))>d\right\}
$$
tends to $0$ with $k\to\infty$. We have
%
$$
|I_1|\leq\int_{K\setminus f_k^{-1}(N\setminus D_1)}|J_{f_k}-J_f|+
\int_{K\setminus f_k^{-1}(N\setminus D_1)}|J_{f}|.
$$
The first term in the above sum tends to $0$ with $k\to\infty$ thanks to the convergence of $J_{f_k}$ to $J_f$ in $L^1$; 
the second one also  tends to $0$, because the measure of the set of integration tends to $0$, as proven above. Thus, for $k$ 
sufficiently large, the term $I_2$ must be positive. This proves (c) and completes the proof of Lemma~\ref{l2}.
\end{proof}

Now we are ready to complete the proof of Proposition~\ref{l11}.
We claim that there is an open set $D'$ such that $\overbar{D'}\Subset D_3\setminus D_2$ is disjoint from the compact set
$\tf(\overbar{B})$. Indeed, since $\tf(\partial B)=f(\partial B)\subset D_1$, we have
$\tf(\overbar{B})\cap (D_3\setminus D_2)\subset \tf(B)$
and it suffices to observe that
$$
\Vol(\tf(\overbar{B})\cap (D_3\setminus D_2))\leq \Vol(\tf(B))\leq \int_B J_{\tf}<\Vol(D_3\setminus D_2).
$$
If $p$ is the common center of the balls $D_1$, $D_2$ and $D_3$, then the inverse of the exponential map
$\exp^{-1}_p$ maps the balls $D_1$, $D_2$ and $D_3$ onto the balls $B(0,2d)$, $B(0,3d)$ and $B(0,d_N/2)$. We can also assume that the set 
$D'$ is such that $\exp_p^{-1}(D')$ is a Euclidean ball $B'$ whose closure is contained in $B(0,d_N/2)\setminus B(0,3d)$.

\begin{figure}[h]

\begin{tikzpicture}[>=latex']
\def\centerarc[#1](#2)(#3:#4:#5)
    { \draw[#1] ($(#2)+({#5*cos(#3)},{#5*sin(#3)})$) arc (#3:#4:#5); }

\node at (0,2.3) {};
\filldraw[color=black!45!white,draw=darkgray]
(0.31,-1) to [out=120, in =-70]  (0.25, -0.9) to [out=0,in=200] (0.6,-0.8) to [out=20,in=140] (0.5,-1) to [out=-40, in =-10] (1,-1.3) to [out = 170, in =120] (0.5,-1.1) to [out=-60, in=0] (0.7,-1.3) to [out=180,in=-20] (0.31,-1)--cycle;
\shadedraw[top color=black!30!white, bottom color = black!5!white, opacity=0.5] (0,0) to [out=-90, in=175] (3,-2.2) to [out=-5, in=190] (6,-2) to [out=10, in =160] (9,-1.5) to [out=-20,in=-90] (12,0) to  [out=90, in=20] (9,1) to [out=200, in=-20] (5.5,1.5) to [out=160, in=0] (3, 2) to [out=180, in=90] (0,0) -- cycle;
\filldraw[fill=white] (4,0) to [out=-30, in=210] (6,0) to [out=165, in=15] (4,0) -- cycle;
\draw[thick] (3.75, 0.2) to [out=-50,in=150] (4,0) to [out=-30, in=210] (6,0) to [out=30, in=50] (6.15,0.1);
\filldraw[fill=white] (9.5,0) to [out=-30, in=210] (10.5,0) to [out=165, in=15] (9.5,0) -- cycle;
\draw[thick] (9.35, 0.1) to [out=-50,in=150] (9.5,0) to [out=-30, in=210] (10.5,0) to [out=30, in=50] (10.56,0.07);

\draw[thick,dotted,draw=darkgray] (3,-2.2) to [out=140,in=-90] (2,0) to [out=90,in=-140] (3, 2);

\shadedraw[top color=black!45!white, bottom color = black!5!white,draw=gray,opacity=0.5] (0,0) to [out=-90, in=175] (3,-2.2) to [out=65,in=-90] (3.5,0) to [out=90,in=-65] (3, 2) to [out=180, in=90] (0,0) -- cycle;
\node  at (2,-1.5) {$D_3$};

\filldraw[color=black!35!white ,draw=darkgray,opacity=0.7] (1.4,0) to [out=0, in=210] (2,0.6) to [out=30, in=185] (3.5, 1.2)  to [out=5, in=-90] (3.7, 1.7) to [out=90, in = 120] (3.9,1.1)  to [out=-60, in=170] (6,1.2)  to [out=-10, in=20] (5,0.9)  to [out=200, in =-60] (3.8,0.9) to [out=120,in =20] (2,0.5) to [out=200, in =80] (1.7,-0.2) to [out=260, in=170] (2.7,-0.5) to [out=-10,in=90] (3,0) to [out=-90,in=100] (2.9, -0.5) to [out=-80, in=30] (4,-0.6) to [out=210, in =60] (3,-0.9) to [out=240, in=5] (5,-1.5) to [out=185, in=-10] (3,-1.3) to [out=170, in =-40] (2.8, -0.7) to [out=140, in=0] (1.4,-0.6) to [out=180,in =60] (0.31,-1) to [out=120, in =-70]  (0.25, -0.9) to [out=50,in=240] (1,-0.2) to [out=60,in=180] (1.2,0.5) to [out=0,in=180] (1.4,0) --  cycle;
\draw[->,thin,darkgray] (3.4,-2.3) -- (2.96, -0.75);
\node[below] at (3.4,-2.3) {$\tilde{f}(\overbar{B})$};

\draw[thick] (1.15,-0.1) to [out=90, in=120] (1.2,0) to [out=-60,in=-150] (1.5,-0.1) to [out=30,in=30] (1.4,-0.3) to [out=-150,in=0] (1.4, -0.4) to [out=180,in=20] (1.2,-0.4)	to [out=200,in=180] (1.15,-0.2) to [out=0,in=-90] (1.25,-0.2) to [out=90,in=-90] (1.15,-0.1) -- cycle ;
\draw[->,thin,darkgray] (1.2,-2.1) -- (1.4, -0.4);
\node[below] at (1.2,-2.1) {$\tilde{f}(\partial B)$};

\node at (0.95,-0.4) {$\scriptstyle D_{1}$};
\centerarc[thin](1.3,-0.3)(-150:170:0.4);
\draw[help lines] (1.3,-0.3) circle [radius=0.4];
\node[darkgray] at (2.1,-0.3) {$\scriptstyle D_2$};

\draw[help lines] (1.3,-0.3) circle [radius=0.8];
\centerarc[thin](1.3,-0.3)(10:350:0.8);
\filldraw[color=white ,draw=gray,opacity=0.6] (1.3,1.1) circle [x radius=0.4, y radius=0.3];
\node at (1.3,1.1) {$\scriptstyle D'$};
\end{tikzpicture}

\caption{The image of $\tf(\partial B)$ is contained in $D_1$; there exists an open set  $D'\subset D_3$ with closure disjoint from $\tf(\overbar{B})$.}
\label{f1}
\end{figure}
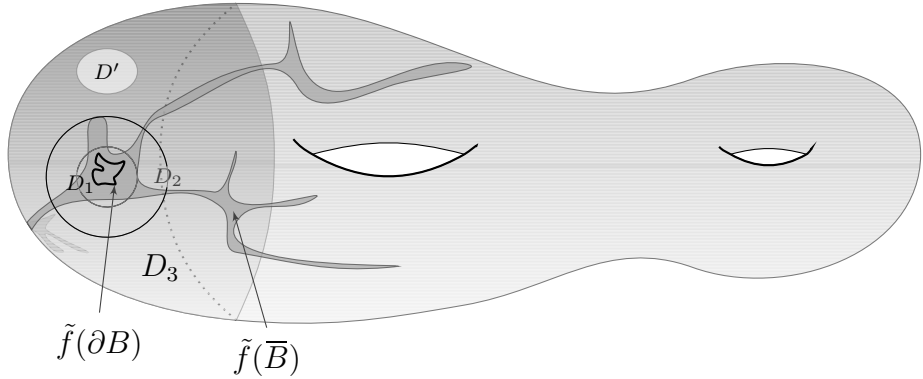

We claim that there is a Lipschitz retraction 
$$
R\colon N\setminus \overbar{D'}\to\overbar{D_1}.
$$
Let $\Phi\colon \overbar{B(0,d_N/2)} \to \bbbs^n$ be a Lipschitz map onto the unit sphere in $\bbbr^n$ such that
\begin{itemize}
\item $\Phi$ maps $\partial B(0,d_N/2)$ onto a fixed point $x_o\in\bbbs^n$, on the equator of $\bbbs^n$,
\item $\Phi$ is a diffeomorphism of $B(0,d_N/2)$ onto $\bbbs^n \setminus \{x_o\}$,
\item $\Phi$ maps balls $\overbar{B(0,2d)}$ and $\overbar{B'}$ onto polar caps $\mathscr{C}_S$ and $\mathscr{C}_N$ around the south and the north pole respectively.
\end{itemize}
Then
$$
\Phi\circ \exp_p^{-1}\colon \overbar{D_3} \to \bbbs^n
$$
maps $\partial D_3$ onto $x_o$, so the map $\Psi:N\to \bbbs^n$ defined by
$$
\Psi(x)=\begin{cases} \Phi\circ \exp_p^{-1}(x) & \text{ if }x\in \overbar{D_3},\\
x_o&\text{ if }x\in N\setminus \overbar{D_3}
\end{cases}
$$
is Lipschitz.
Let $p\colon \bbbs^n\setminus \mathscr{C}_N \to \mathscr{C}_S$ be the retraction along meridians.

$\Psi$ maps $\overbar{D'}$ onto $\mathscr{C}_N$, so $\Psi$ maps $N\setminus \overbar{D'}$ onto $\bbbs^n\setminus \mathscr{C}_N$ 
and hence $p\circ\Psi\colon N\setminus \overbar{D'}\to\mathscr{C}_S$ is well defined and Lipschitz. Note also that $\Psi$ maps 
$\overbar{D}_1$ onto $\mathscr{C}_S$ in a diffeomorphic way, so $\big( \Psi|_{\overbar{D_1}}\big)^{-1}$  
maps $\mathscr{C}_S$ back onto $\overbar{D_1}$. Hence the map
$$
R= \big( \Psi|_{\overbar{D_1}}\big)^{-1}\circ p\circ  \Psi \colon N\setminus \overbar{D'} \to \overbar{D_1}
$$
is well defined and Lipschitz.

If $x\in\overbar{D_1}$, then $\Psi(x)\in\mathscr{C}_S$, so $p(\Psi(x))=\Psi(x)$ and hence $R(x)=x$. 
That means $R\colon N\setminus \overbar{D'}\to \overbar{D_1}$ is a Lipschitz retraction.

Since $\tf$ maps $\overbar{B}$ into $N\setminus \overbar{D'}$, the map 
$R\circ \tf\colon \overbar{B}\to N$ is in $C^0\cap W^{1,n}$. Also $\tf |_{\partial B}=R\circ\tf |_{\partial B}$, 
because $\tf(\partial B)=f(\partial B)\subset D_1$ and $R$ is identity on $\overbar{D_1}$.

Now, note that the mappings $\tf:B\to N$ and $R\circ\tf:B\to N$ are both continuous, in $W^{1,n}$, and have the same boundary values. We can thus define a 
mapping $F:\bbbs^n\to N$ in such a way that on the upper hemisphere (diffeomorphic to $B$) $F$ coincides with $\tf$, and on the lower 
hemisphere -- with $R\circ \tf$, see Figure \ref{f3}. Obviously, $F$ constructed this way is continuous and in $W^{1,n}$. It is also 
not onto $N$, since $D'$ is disjoint with its image. 

Recall that by a result of Schoen and Uhlenbeck \cite{schoenu1,schoenu2}, smooth mappings $C^\infty(\bbbs^n,N)$ are dense in $W^{1,n}(\bbbs^n,N)$
(see also a survey paper \cite{Hajlasz} and references therein). In fact, since $F$ is continuous, we can find a sequence
$F_k\in C^\infty(\bbbs^n,N)$ that converges to $F$ both uniformly and in the $W^{1,n}$ norm. Hence for $k\geq k_0$, $F_k$ is not surjective so its degree equals zero,
and thus $\int_{\bbbs^n}J_{F_k}=0$. Since the integral of the Jacobian is continuous in the $W^{1,n}$ norm (by H\"older's inequality), it follows that
$$
\int_{B}J_{\tf}-\int_{B}J_{R\circ \tf}=\int_{\bbbs^n} J_F = \lim_{k\to\infty} \int_{\bbbs^n} J_{F_k} = 0.
$$
That is 
\begin{equation}\label{eq4}
\int_{B}J_{\tf}=\int_{B} J_{R\circ \tf}.
\end{equation}
The mappings $\tf$ and $R\circ \tf$ coincide in $\tf^{-1}(D_1)$ (and thus their Jacobians $J_{\tf}$ and $J_{R\circ \tf}$ are equal there), 
which, together with \eqref{eq4}, implies that
\begin{equation}
\label{eq5}
\int_{B\cap\tf^{-1}(N\setminus D_1)}J_{\tf}=\int_{B\cap\tf^{-1}(N\setminus D_1)}J_{R\circ \tf}.
\end{equation}
Note that $R\circ \tf$ maps the set $\tf^{-1}(N\setminus D_1)$ onto the boundary of the ball $\overbar{D_1}$. Hence the Jacobian of
$R\circ \tf$ equals zero on $\tf^{-1}(N\setminus D_1)$ proving that the right hand side of \eqref{eq5} is zero. However,
the left hand side is positive by Lemma~\ref{l2}, (c). This yields the desired contradiction, that ends the proof of Proposition~\ref{l11}.
\end{proof}

\begin{figure}[h]
\begin{tikzpicture}[scale=0.9,>=latex']


\shadedraw[top color=black!25!white, bottom color=black!5!white,draw=black] (0,0) circle [radius=1.5];
\filldraw[fill=lightgray] (-1.5,0) arc (180:360:1.5 and 0.5) arc (0:180:1.5) -- cycle;
\draw[thick] (-1.5,0) arc (180:360:1.5 and 0.5);
\draw[dotted] (1.5,0) arc (0:180:1.5 and 0.5);
\draw[thick,->] (1.2,1.2) to [out=20,in=160] (3.4,1.2);
\draw[->] (1,-0.38) to [out=-10,in=-170] (3.92,-0.3);
\node at (2.2,1.7) {$\tilde{f}$};
\draw[thick,->] (1.2,-1.2) to [out=-20,in=-160] (3.4,-1.2);
\node at (2.2,-1.75) {$R\circ \tilde{f}$};
\begin{scope}[shift={(3,0)},scale=0.8]
\filldraw[color=black!20!white,draw=gray]
(0.31,-1) to [out=120, in =-70]  (0.25, -0.9) to [out=0,in=200] (0.6,-0.8) to [out=20,in=140] (0.5,-1) to [out=-40, in =-10] (1,-1.3) to [out = 170, in =120] (0.5,-1.1) to [out=-60, in=0] (0.7,-1.3) to [out=180,in=-20] (0.31,-1)--cycle;
\shadedraw[top color=black!35!white, bottom color = black!5!white, opacity=0.5] (0,0) to [out=-90, in=175] (3,-2.2) to [out=-5, in=190] (6,-2) to [out=10, in =160] (9,-1.5) to [out=-20,in=-90] (12,0) to  [out=90, in=20] (9,1) to [out=200, in=-20] (5.5,1.5) to [out=160, in=0] (3, 2) to [out=180, in=90] (0,0) -- cycle;


\filldraw[fill=white] (4,0) to [out=-30, in=210] (6,0) to [out=165, in=15] (4,0) -- cycle;
\draw[thick] (3.75, 0.2) to [out=-50,in=150] (4,0) to [out=-30, in=210] (6,0) to [out=30, in=50] (6.15,0.1);
\filldraw[fill=white] (9.5,0) to [out=-30, in=210] (10.5,0) to [out=165, in=15] (9.5,0) -- cycle;
\draw[thick] (9.35, 0.1) to [out=-50,in=150] (9.5,0) to [out=-30, in=210] (10.5,0) to [out=30, in=50] (10.56,0.07);



\filldraw[color=black!40!white ,draw=darkgray,opacity=0.7] (1.4,0) to [out=0, in=210] (2,0.6) to [out=30, in=185] (3.5, 1.2)  to [out=5, in=-90] (3.7, 1.7) to [out=90, in = 120] (3.9,1.1)  to [out=-60, in=170] (6,1.2)  to [out=-10, in=20] (5,0.9)  to [out=200, in =-60] (3.8,0.9) to [out=120,in =20] (2,0.5) to [out=200, in =80] (1.7,-0.2) to [out=260, in=170] (2.7,-0.5) to [out=-10,in=90] (3,0) to [out=-90,in=100] (2.9, -0.5) to [out=-80, in=30] (4,-0.6) to [out=210, in =60] (3,-0.9) to [out=240, in=5] (5,-1.5) to [out=185, in=-10] (3,-1.3) to [out=170, in =-40] (2.8, -0.7) to [out=140, in=0] (1.4,-0.6) to [out=180,in =60] (0.31,-1) to [out=120, in =-70]  (0.25, -0.9) to [out=50,in=240] (1,-0.2) to [out=60,in=180] (1.2,0.5) to [out=0,in=180] (1.4,0) --  cycle;

\draw[gray,fill] (1.3,-0.3) circle [radius=0.025];

\draw[thick] (1.15,-0.1) to [out=90, in=120] (1.2,0) to [out=-60,in=-150] (1.5,-0.1) to [out=30,in=30] (1.4,-0.3) to [out=-150,in=0] (1.4, -0.4) to [out=180,in=20] (1.2,-0.4)	to [out=200,in=180] (1.15,-0.2) to [out=0,in=-90] (1.25,-0.2) to [out=90,in=-90] (1.15,-0.1) -- cycle ;
\draw[thick,gray] (1.3,-0.3) circle [radius=0.5];
\filldraw[thick,color=black!2!white,draw=lightgray] (1.3,1.1) circle [x radius=0.4, y radius=0.3];
\node[scale=0.7] at (1.3,1.1) {$D'$};
\end{scope}
\draw[->] (1,-0.38) to [out=-10,in=-170] (3.92,-0.3);

\end{tikzpicture}
\caption{Mapping $F$.}
\label{f3}
\end{figure}
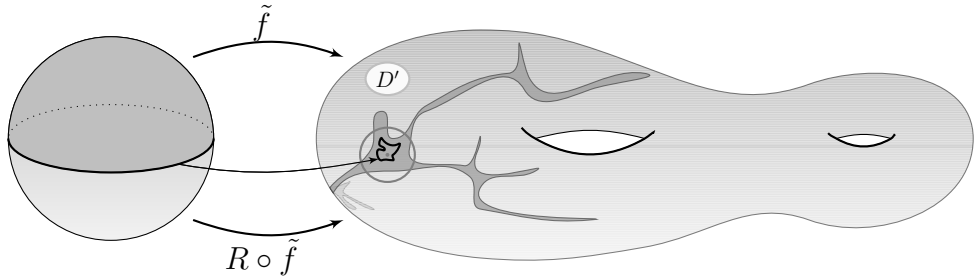

With the help of Proposition~\ref{l11} we can now complete the proof of Theorem~\ref{from ball}.

Fix $x\in \bbbb^n$. 
According to Fubini's theorem, for almost all $r$ such that $0<r\leq\dist(x,\partial \bbbb^n)$, the function $f$ restricted to the 
sphere $S(x,r)$ belongs to the Sobolev space $W^{1,n}(S(x,r))$. Since $n>n-1=\dim S(x,r)$, $f$ restricted to $S(x,r)$ is H\"older 
continuous with exponent $1/n$ and Morrey's inequality (see Corollary \ref{cor:morrey} in the Appendix) yields 
\begin{equation}
\diam f(S(x,r)) \leq \C_M \Bigg(r\int_{S(x,r)} |Df|^n\Bigg)^{\frac{1}{n}},
\end{equation}
where the constant $\C_M$ depends on $n$ only (see \eqref{morrey euc}).

In fact, one can assume more: that (after a proper choice of a representative) $f$ is continuous on $S(x,r)$ for almost all $(x,r)\in\bbbb^n\times(0,\infty)$ such that $0<r\leq \dist(x,\partial\bbbb^n)$ and on almost all lines parallel to the coordinate directions, see Corollary  \ref{cor:12} in Appendix. 
From now on we shall assume that $f$ is such a representative; let $X$ be the full measure subset of $\bbbb^n$ such that $f$ is 
continuous on a.e. sphere centered at $x\in X$ and contained in $\bbbb^n$.

Recall also that the constants $\A_N$ and $\B_N$ were defined in the statement of Theorem~\ref{from ball}.

Fix $x_o\in X$ and $R>0$ such that $B(x_o,2R)\subset\bbbb^n$, and fix $r\in (0,R)$. Let
$$
T=\Big\{ t\in (0,2R):\, t\int_{S(x_o,t)}|D f|^n\geq 2\A_N\Big\}\, .
$$
Then
$$
\A_N  > \int_{B(x_0,2R)} |Df|^n = 
\int_0^{2R} \Big(\int_{S(x_o,t)} |Df|^n\Big)\, dt
\geq \int_T\frac{1}{t}\Big( t\int_{S(x_o,t)} |Df|^n\Big)\, dt 
 \geq |T|\cdot\frac{1}{2R}\cdot 2 A_N,
$$
so $|T|<R$. Hence we can find $\rho\in (R,2R)$ (in the complement of $T$) such that
\begin{itemize}
\item[(a)] $f\in C^{0,1/n}\cap W^{1,n}(S(x_0,\rho),N)$,
\item[(b)]
$\rho\int_{S(x_o,\rho)} |Df|^n<2\A_N$.
\end{itemize}
This and Morrey's inequality (Corollary~\ref{cor:morrey}) yield
$$
\diam f(S(x_o,\rho))\leq \C_M\Big(\rho\int_{S(x_o,\rho)}|Df|^n\Big)^{1/n}<
\C_M(2\A_N)^{1/n}=d_N/60\, .
$$
Let $d=\diam f(S(x_o,\rho))$. Note that $d<d_N/60<d_N/10$.

Let $D_1$, $D_2$ and $D_3$ be concentric balls in $N$ of radii $2d$, $3d$ and $d_N/2$ such that
$f(S(x_o,\rho))\subset D_1$.

Note that $D_3\setminus D_2$ contains a ball of radius $d_N/10$, thus
$\Vol(D_3\setminus D_2)\geq \B_N$. On the other hand, the pointwise estimate
$J_f\leq |Df|^n$ yields
$$
\int_{B(x_o,\rho)} J_f\leq \int_{B(x_o,\rho)} |Df|^n<\B_N\leq\Vol(D_3\setminus D_2).
$$
Overall, $B(x_o,\rho)$ satisfies assumptions of Proposition~\ref{l11}, by which almost all points of
$B(x_o,\rho)$ are mapped into $\overbar{D_2}$.
Hence
$$
\osc_{B(x_0,\rho)} f \leq \diam\overbar{D_2} = 6d=6\osc_{S(x_o,\rho)} f<d_N/10.
$$
Here, as before, we understand the oscillation on the ball as the essential oscillation.

In particular, for almost all $t\in (0,\rho)$, $f$ restricted to $S(x_o,t)$ is H\"older continuous and almost all
points of $S(x_o,t)$ are mapped into $\overbar{D_2}$ so $f(S(x_o,t))\subset \overbar{D_2}$.
For such $t\in (0,\rho)$,
$$
\diam f(S(x_o,t))\leq\diam\overbar{D_2}<d_N/10.
$$
That means all such balls $B(x_o,t)$ satisfy the assumptions of Proposition~\ref{l11} and thus Proposition~\ref{l11} and Morrey's inequality yield
$$
\osc_{B(x_o,t)} f\leq 6\osc_{S(x_o,t)} f \leq 6 \C_M\Big( t\int_{S(x_o,t)} |Df|^n\Big)^{1/n}
$$
for almost all $t\in (0,\rho)$. Since $(0,R)\subset (0,\rho)$, it is also true for almost all $t\in (0,R)$.
Integrating this inequality yields (recall that $0<r<R$)
\begin{align*}
\left(\osc_{B(x_o,r)} f\right)^n &\log(R/r) 
 = \int_r^R \left(\osc_{B(x_o,r)} f\right)^n\frac{dt}{t}
\leq \int_r^R \left(\osc_{B(x_o,t)} f\right)^n\frac{dt}{t} \\
& \leq (6\C_M)^n \int_r^R \Bigg(\int_{S(x_o,t)} |Df|^n\Bigg)\, dt 
= (6 \C_M)^n \int_{B(x_0,R)\setminus B(x_o,r)} |Df|^n.
\end{align*}
 This proves \eqref{osc1r}. Since oscillation on balls is essential oscillation, we do 
not know yet if $f$ has a continuous representative.

We shall write, for short, that a triple $(x,r,R)$ satisfies condition (T), 
if $B(x,2R)\subset\bbbb^n$ and $0<r<R$.

Whenever $x\in X$ and $(x,r,R)$ satisfies (T), the inequality 
\begin{equation}
\label{A}
\Big(\osc_{B(x,r)}f\Big)^n\leq \frac{C(n)}{\log(R/r)} \int_{B(x,R)} |Df|^n
\end{equation}
holds with $C(n)=(6\C_M)^n$.  However, the oscillation $\osc_{B(x,r)}f$ in \eqref{A} is understood as the essential oscillation, 
that is there is a subset $N_{x,r,R}\subset B(x,r)$ of measure zero  such that  for any $u,v\in B(x,r)\setminus N_{x,r,R}$ there holds the pointwise estimate
\begin{equation}
\label{B}
d(f(u),f(v))^n\leq \frac{C(n)}{\log(R/r)} \int_{B(x,R)} |Df|^n.
\end{equation}

Let us now choose a countable subset $\tilde{X}\subset X$, $\tilde{X}$ dense in $\bbbb^n$, and let 
$$
N=\bigcup_{(x,r,R)} N_{x,r,R},
$$
where the sum runs over all triples $(x,r,R)$ satisfying conditions (T) and such that $x\in\tilde{X}$ and $r,R\in\mathbb{Q}$. Obviously $N$ has measure zero.

Fix $u,v\in \bbbb^n\setminus N$.
The set of all triples $(x,r,R)$ satisfying (T) and such that $u,v\in B(x,r)$ is open in 
$\bbbb^n\times (0,\infty)\times (0,\infty)$, therefore to each such triple $(x,r,R)$ we 
can find another triple $(x',r',R')$ that satisfies (T), is arbitrarily close to $(x,r,R)$, 
we have $u,v\in B(x',r')$, $x'\in \tilde{X}$ and $r,R\in\mathbb{Q}$.

Then $u,v\in B(x',r')\setminus  N_{x',r',R'}$, thus 
$$
d(f(u),f(v))^n\leq \frac{C(n)}{\log(R'/r')} \int_{B(x',R')} |Df|^n,
$$
but since $(x',r',R')$ is arbitrarily close to $(x,r,R)$, the estimate \eqref{B} holds as well. 
Thus we established \eqref{B} for all $u,v\in \bbbb^n\setminus N$ and any $(x,r,R)$ satisfying (T) and such that $u,v\in B(x,r)$.

Let now $p$ and $\rho$ denote the center and radius of $\bbbb^n$ respectively. 
Fix  $\eps\in (0,\rho/2)$ and assume that $u,v\in B(p,\rho-2\eps)\setminus N$ and $|u-v|<\eps$. 
Taking in \eqref{B}, $x=(u+v)/2$, $r=|u-v|$, $R=\eps$, we get
that $u,v\in B(x,r)\setminus N$ and $(x,r,R)$ satisfies condition (T), so
$$
 d(f(u),f(v))^n\leq \frac{C(n)}{\log(\eps/|u-v|)} \int_{B(x,R)} |Df|^n\leq 
\frac{C(n)}{\log(\eps/|u-v|)}\int_{\bbbb^n} |Df|^n.
 $$
This shows that $f$ is uniformly continuous on $B(p,\rho-2\eps)\setminus N$, and since $\eps>0$ can be chosen arbitrarily small, 
$f$ is locally uniformly continuous on $\bbbb^n\setminus N$. 
 
The set $\bbbb^n\setminus N$ is dense in $\bbbb^n$, hence 
$f|_{\bbbb^n\setminus N}$ admits a unique continuous extension to~$\bbbb^n$. This extension agrees with $f$ on the set 
$\bbbb^n\setminus N$ of full measure, thus it indeed is a representative of $f$.
 
Now, by continuity of $f$, the estimate \eqref{B}  holds not only outside $N_{x,r,R}$, but on the whole $B(x,r)$ and 
the inequality \eqref{A} follows for all $(x,r,R)$ satisfying (T). This ends the proof of Theorem~\ref{from ball}.
\end{proof}

\section{Orlicz-Sobolev mappings with positive Jacobian}\label{sec:3}
In this Section we give the example of a discontinuous mapping $f:\bbbs^n\to\bbbs^n$ such that $J_f>0$ a.e. (and thus $f$ has finite distortion), 
with $Df\in L^n \Log^{-1}$. The example is a slight simplification of the example of a mapping with infinite degree, given in \cite[Section~3.4]{HIMO}.

Let
$$
S^\beta_\alpha=\{(z\sin\theta,\cos \theta)~~:~~z\in\bbbs^{n-1},\,\alpha\leq\theta\leq\beta\}\subset\bbbs^n
$$
be the spherical slice bounded by latitude spheres $\theta=\alpha$ and $\theta=\beta$ (with $\theta=0$ denoting the north and $\theta=\pi$ 
the south pole of $\bbbs^n$). One can easily check using integration in spherical coordinates that
$$
\Vol(S^\beta_\alpha)\leq\int_\alpha^\beta n\omega_n \sin^{n-1}\theta\, d\theta\leq \omega_n (\beta^n-\alpha^n).
$$

The mapping $f^\beta_\alpha:S^\beta_\alpha\to \bbbs^n$, 
$$
f^\beta_\alpha(z\sin\theta,\cos\theta)=\Big(z\sin\frac{(\theta-\alpha)\pi}{\beta-\alpha},\cos \frac{(\theta-\alpha)\pi}{\beta-\alpha}\Big),
$$
stretches the slice $S^\beta_\alpha$ to the whole sphere $\bbbs^n$, mapping the latitude sphere $\theta=\alpha$ into the north pole and $\theta=\beta$ 
into the south pole. Obviously, $f^\beta_\alpha$ preserves the orientation, thus $J_{f^\beta_\alpha}>0$.

Likewise, if $\sigma:\bbbs^{n-1}\to \bbbs^{n-1}$ is a fixed reflection (i.e. an orientation reversing isometry), the mapping 
$$
\tilde{f}^\beta_\alpha(z\sin\theta,\cos\theta)=\Big(\sigma(z)\sin\frac{(\beta-\theta)\pi}{\beta-\alpha},\cos \frac{(\beta-\theta)\pi}{\beta-\alpha}\Big)
$$
stretches $S^\beta_\alpha$ to the whole sphere $\bbbs^n$, mapping this time $\theta=\alpha$ into the south pole and $\theta=\beta$ into the north pole, 
and the mapping is again orientation preserving (the change in orientation in the latitude coordinate is compensated by $\sigma$ reversing the 
orientation in the equatorial, i.e. longitude coordinates), hence $J_{\tilde{f}^\beta_\alpha}>0$.

If $\beta\leq \pi/2$, i.e. $S^\beta_\alpha$ is contained in the upper hemisphere, the map $f_\beta^\alpha$ stretches in the equatorial (latitude) directions by the factor $\left|\sin \frac{(\theta-\alpha)\pi}{\beta-\alpha}/\sin\theta\right|$ and in the longitude direction by the factor $\pi/(\beta-\alpha)$, thus 
$$
|Df^\beta_\alpha|\lesssim\max\Big\{\frac{1}{\sin\alpha}, \frac{\pi}{\beta-\alpha}\Big\}\, .
$$ 
An analogous estimate holds for $D\tilde{f}^\beta_\alpha$.

Now, let us fix $\theta_k=2^{-k^2}\pi$ for $k=0,1,\ldots$ and define $f:\bbbs^n\to\bbbs^n$,
$$
f(z\sin\theta,\cos\theta)=\begin{cases} f^{\theta_{k-1}}_{\theta_k} &\text{ if } \theta_{k}\leq \theta\leq \theta_{k-1}\text{ and }k\geq 1\text{ is odd},\\
\tilde{f}^{\theta_{k-1}}_{\theta_k} &\text{ if } \theta_{k}\leq \theta\leq \theta_{k-1}\text{ and }k\geq 1\text{ is even}.
\end{cases}
$$
The above mapping stretches each of the slices $A_k=S^{\theta_{k-1}}_{\theta_k}$ onto the whole sphere in an orientation-preserving way 
(thus $J_f>0$ a.e.), mapping the longitude spheres $\theta=\theta_k$ alternately into the north and south poles. Obviously, $f$ has a non-removable 
discontinuity at the north pole.  However, if $\theta\in (\theta_{k-1},\theta_k)$, $k\geq 2$, we have 
$$
|Df(z\sin\theta,\cos\theta)|\lesssim \frac{1}{\sin\theta_{k-1}}+\frac{\pi}{\theta_{k-1}-\theta_k}\lesssim 2^{(k-1)^2}.
$$
If $k=1$, $f^{\theta_o}_{\theta_1}=f^\pi_{\pi/2}$ is Lipschitz, so we can estimate the norm of its derivative by a constant.
Therefore, for $P(t)=t^n/\log(e+t)$, 
\begin{equation}
\begin{split}
\int_{\bbbs^n}P(|Df|)&=\sum_{k=1}^\infty \int_{A_k} P(|Df|)\leq \int_{A_1} P(|Df|)+\sum_{k=2}^\infty |A_k | P(2^{(k-1)^2})\\
&\lesssim 1+\sum_{k=2}^\infty (2^{-(k-1)^2 n}-2^{-k^2 n}) \frac{2^{(k-1)^2 n}}{(k-1)^2}\\
&\leq 1+\sum_{k=2}^\infty \frac{1}{(k-1)^2}<\infty,
\end{split}
\end{equation}
which proves that $|Df|\in L^n\Log^{-1}$.


\section{Appendix}\label{sec:5}
\subsection{Morrey's inequality on manifolds}
\label{sec:mor}

The following inequality for real valued functions was proved in \cite[Lemma 4.10.1]{IwaniecM}.
\begin{lemma}
\label{osclemma}
If $S(r)$ is an $(n-1)$-dimensional Euclidean sphere of radius $r$ and $f\in W^{1,n}(S(r))$, then 
$f$ has a $C^{0,1/n}$-H\"older continuous representative
which satisfies
\begin{equation}
\label{morrey euc}
\osc_{S(r)} f=\sup_{x,y\in S(r)}|f(x)-f(y)|\leq \C_M\,\Bigg( r\int_{S(r)}|Df|^n\Bigg)^{1/n},
\quad
\text{where}
\quad
\C_M=\frac{(n-1)\pi}{(n\omega_n)^{1/n}}\, .
\end{equation}
\end{lemma}
The proof given in \cite{IwaniecM} shows that this inequality is true with respect to the operator norm $|Df|$ of the derivative $Df$.

As a corollary, we obtain that a similar inequality, with the same constant, is true for mappings into any compact Riemannian manifold $N$ without boundary.

\begin{corollary}
\label{cor:morrey}
Let $S(r)$ be an $(n-1)$-dimensional Euclidean sphere of radius $r$ and let $N$ be a compact Riemannian manifold without 
boundary. If $f\in W^{1,n}(S(r),N)$, then its $C^{0,1/n}$-H\"older continuous representative satisfies
\begin{equation}
\label{eq:morrey}
\diam f(S(r))\leq \C_M\, \Bigg( r\int_{S(r)}|Df|^n\Bigg)^{1/n}.
\end{equation}
\begin{proof}
Every separable metric space $(X,\rho)$ admits an isometric embedding into $\ell^\infty$. Indeed,  if $x_o\in X$ is a fixed point and 
$\{x_i\}_{i=1}^\infty\subset X$ is a countable dense subset, then it is easy to see that the mapping 
$$
X\ni x\longmapsto \{\rho(x,x_i)-\rho(x_i,x_o)\}_{i=1}^\infty\in\ell^\infty
$$
is an isometric embedding. 

The manifold $N$ is a metric space with respect to the Riemannian distance $d$. Let $\kappa\colon (N,d)\to \ell^\infty$ 
be an isometric embedding. If $f\in W^{1,n}(S(r),N)$, then $\kappa_i \circ f\in W^{1,n}(S(r))$ is a real valued function. 
Since the function $\kappa_i$ is $1$-Lipschitz on $(N,d)$, it easily follows that 
$$
|D(\kappa_i\circ f)|\leq |Df| \quad \text{almost everywhere.}
 $$
 Hence for any $x,y\in S(r)$, Lemma \ref{osclemma} yields 
 $$
 |\kappa_i\circ f(x)-\kappa_i\circ f(y)|\leq \C_M\, \Bigg( r\int_{S(r)}|Df|^n\Bigg)^{1/n}
 $$
 so upon taking the supremum over $i\in\N$ we have
 $$
 d(f(x),f(y))=\|\kappa\circ f(x)-\kappa\circ f (y)\|_\infty\leq \C_M\, \Bigg( r\int_{S(r)}|Df|^n\Bigg)^{1/n}
 $$
 and the result follows.
\end{proof}
\end{corollary}

\subsection{Choosing representatives of Sobolev functions}
The following application of Fubini's theorem to Sobolev functions can be reproduced (and strengthened) using the notions of 
Sobolev capacities and quasicontinuous representatives; we prove it,  
however, more directly.
\begin{lemma}
\label{continuity a.e.}
Let $f\in W^{1,p}(\R^n)$ with $p>n-1$. Then there exists a full measure subset  $A\subset\R^n\times (0,\infty)$  such that 
$f$ (or, more precisely, a certain representative of $f$) is H\"older continuous on every sphere $S(x,r)$ with $(x,r)\in A$.
In other words, $f$ is H\"older continuous on a.e. sphere in $\R^n$.
\end{lemma}
\begin{proof}
Let $(f_k)\in C^{\infty}(\R^n)$ be a sequence converging to $f$ in $W^{1,p}(\R^n)$. Then, for any fixed $x\in\R^n$,
\begin{equation}
\label{fubini}
\begin{split}
\int_{\R^n} &\big(|f-f_k|^p +|D f -D f_k|^p\big)\\
&=\int_0^\infty\int_{S(x,r)}\big(|f-f_k|^p +|D f -D f_k|^p\big)dr\to 0\quad\text{ as }k\to\infty.
\end{split}
\end{equation}
with the convergence rate obviously uniform in $x$. Therefore,  writing 
$$
{F_k(x,r)=\int_{S(x,r)}\big(|f-f_k|^p +|D f -D f_k|^p\big)}
$$ 
and integrating \eqref{fubini} over $\bbbr^n$ with respect to the measure $d\mu(x)=e^{-|x|^2}\, dx$
\begin{equation}
\label{937}
\int_{\R^n}\int_0^\infty F_k(x,r) dr\, d\mu(x) \to 0 \quad \text{ as }k\to\infty,
\end{equation}
meaning that $F_k\to 0$ in $L^1(\R^n\times (0,\infty),d\mu\otimes dr)$ with $k\to\infty$. 
We had to integrate with respect to a finite measure $\mu$ on $\bbbr^n$ as otherwise the integral \eqref{937} would be equal infinity.
Passing to a subsequence (denoted again by $F_k$) we may 
assume that $F_k\to 0$ a.e. in $\R^n\times(0,\infty)$, thus there is a set $A$ of full measure in $\R^n\times(0,\infty)$ such that
$$
\int_{S(x,r)}\big(|f-f_k|^p +|D f -D f_k|^p\big)\to 0 \quad \text{ when }(x,r)\in A.
$$
Since $p>n-1=\dim S(x,r)$, Morrey's theorem implies that the sequence $f_k$ converges on $S(x,r)$ uniformly
to a H\"older continuous function, thus its pointwise limit is H\"older continuous. 
Choosing the representative of $f$ to be the pointwise limit of $f_k$ (whenever it exists) proves the claim.
\end{proof}
A similar result (through an analogous reasoning) holds also for lines: $f$ has a representative which is (absolutely) 
continuous on almost all lines parallel to the coordinate directions, see e.g. \cite[Theorem 4.21]{EvansGariepy}.
In all cases the representative is given as a pointwise limit of a smooth approximation, therefore we may assume it is the 
same representative as the one in Lemma~\ref{continuity a.e.}.
This in particular implies the following useful corollary.
\begin{corollary}\label{cor:12}
Let $f\in W^{1,p}(\R^n)$ with $p>n-1$. 
Then $f$ has a representative that is absolutely continuous on almost all lines parallel to coordinate directions and 
there exists a full measure subset $X\subset\R^n$ such that for every $x\in X$ the 
function $f$ is H\"older continuous on a.e. sphere centered at $x$.
\end{corollary}

\end{document}